\documentclass[11pt,reqno]{amsart}

\usepackage[a4paper,margin=28mm]{geometry}
\usepackage[T1]{fontenc}
\usepackage[utf8]{inputenc}
\usepackage{mathpazo}
\usepackage{microtype}
\usepackage{amsmath,amssymb,amsthm,mathtools,mathrsfs}
\usepackage{enumitem}
\usepackage{xcolor}
\usepackage{booktabs}
\usepackage{hyperref}
\usepackage[nameinlink,noabbrev,capitalise]{cleveref}

\definecolor{linkblue}{RGB}{29,72,118}
\definecolor{citegreen}{RGB}{35,92,72}
\hypersetup{
  colorlinks=true,
  linkcolor=linkblue,
  citecolor=citegreen,
  urlcolor=linkblue,
  pdftitle={Max-Phase Isometries on Schatten p-Classes Below One},
  pdfauthor={Yacan Liu and Qi Liu}
}

\setlist[itemize]{leftmargin=2em,itemsep=2pt,topsep=4pt}
\setlist[enumerate]{leftmargin=2.3em,itemsep=3pt,topsep=4pt}
\allowdisplaybreaks
\numberwithin{equation}{section}

\newtheorem{theorem}{Theorem}[section]
\newtheorem{proposition}[theorem]{Proposition}
\newtheorem{lemma}[theorem]{Lemma}
\newtheorem{corollary}[theorem]{Corollary}

\theoremstyle{definition}
\newtheorem{definition}[theorem]{Definition}
\newtheorem{example}[theorem]{Example}
\theoremstyle{remark}
\newtheorem{remark}[theorem]{Remark}

\DeclareMathOperator{\Tr}{Tr}
\DeclareMathOperator{\rank}{rank}

\DeclareMathOperator{\spanop}{span}
\newcommand{\C}{\mathbb C}
\newcommand{\R}{\mathbb R}
\newcommand{\T}{\mathbb T}
\newcommand{\Sp}{\mathcal S_p}
\newcommand{\Ffin}{\mathcal F}
\newcommand{\Umin}{\mathcal U_{\min}}
\newcommand{\suppL}{\ell}
\newcommand{\suppR}{r}
\newcommand{\orth}{\perp_{\!s}}
\newcommand{\Mp}{\mathfrak M_p}
\newcommand{\eps}{\varepsilon}
\newcommand{\abs}[1]{\lvert #1\rvert}
\newcommand{\norm}[1]{\lVert #1\rVert}
\newcommand{\ip}[2]{\langle #1,#2\rangle}

\title[On Surjective Max-Phase Isometries]
{On Surjective Max-Phase Isometries of Schatten
	$p$-Classes ($0<p<1$)}

\author{Wenwen Zhang}
\address{School of Mathematics and Statistics, Anqing Normal University,
	Anqing 246133, Anhui, China}
\email{Y26220004@stu.aqnu.edu.cn}

\author{Qi Liu}
\address{School of Mathematics and Statistics, Anqing Normal University,
Anqing 246133, Anhui, China}
\email{liuq67@aqnu.edu.cn}


\begin{document}

\begin{abstract}
	We investigate max-phase isometries on Schatten
	quasi-classes in the range $0<p<1$. Motivated by
	the rigidity theory for smooth normed spaces
	[\emph{Results Math.} \textbf{80} (2025),
	Article 126], we establish its operator-theoretic
	counterpart in the quasi-Banach setting.
	We prove that every surjective max-phase isometry
	on a finite-dimensional Schatten class is a
	standard real-linear Schatten isometry multiplied
	by an even sign-valued function, without any
	continuity assumption. This extends the scope
	of the classical Schatten isometry classification
	[\emph{Proc. Edinburgh Math. Soc.} \textbf{37}
	(1994), 239--242] to nonlinear mappings
	preserving only unsigned two-point metric data.
	Our approach combines support orthogonality,
	a symmetrized rank-one spectral formula, and
	a max-profile uniqueness principle.
	For arbitrary Hilbert spaces, we obtain the
	corresponding classification on finite-rank
	operators, while continuity yields the full-space
	result. We also establish preservation of rank,
	singular values, and two-sided support orthogonality.
\end{abstract}

\maketitle

\section{Introduction}

A recurring theme in the geometry of normed spaces is that surprisingly
limited information about pairs of vectors can determine the structure of a
mapping.  Wigner's theorem provides a classical illustration: preserving
transition probabilities, rather than phases themselves, still forces a
unitary or antiunitary symmetry.  In the setting of real normed spaces,
related questions have motivated the theory of phase isometries and its
max-phase counterpart.  The latter asks what can be recovered from the
single quantity $\max\{\norm{x+y},\norm{x-y}\}$, which records the larger
of the two diagonal lengths associated with a pair of vectors.  Results of
Tan and Liu for real $\ell_p$ spaces with $p>2$ \cite{TanLiu2025}, and the
phase-linearization theorem of Liu, Liu and Tan for smooth strictly convex
spaces \cite{LiuLiuTan2025}, demonstrate the rigidity of this reduced
metric information.  For the closely related theory of phase isometries,
we refer to \cite{HuangTan2020,IlisevicOmladicTurnsek2021}.

This line of inquiry is also connected with the extension problem initiated
by Tingley \cite{Tingley1987}: to what extent is an isometry between unit
spheres induced by a linear isometry of the ambient spaces?  In the
quasi-Banach setting, notable progress has been made for commutative
$L_p$ spaces \cite{Tan2011,LiRen2010}; an operator-theoretic counterpart
for Schatten classes in the Banach range appears in
\cite{FernandezJordaPeralta2020}.  The max-phase problem is different in
a fundamental respect.  Instead of the full distance on a unit sphere,
one is given only the larger of the norms of a sum and a difference.
Consequently, the available two-point information is weaker, and familiar
metric-extension arguments cannot simply be carried over.

The distinction becomes particularly pronounced for Schatten
quasi-classes $\mathcal S_p(H)$ when $0<p<1$.  In the Banach range,
strict convexity and smoothness provide supporting functionals and
first-order perturbation methods that are useful in rigidity arguments.
Below one, the unit ball is nonconvex and such methods lose their natural
foundation.  Moreover, a perturbation acting on the kernel of an operator
may contribute to the $p$th power of the quasi-norm at order $|t|^p$,
which is incompatible with an ordinary first-order linearization.  These
features make the quasi-Banach case a separate geometric problem, rather
than a routine extension of the known Banach-space theory.

The key observation is that the loss of smoothness is accompanied by a
strong equality principle.  The Schatten $p$-quasi-norm satisfies
$\norm{A+B}_p^p\leq\norm A_p^p+\norm B_p^p$, and equality characterizes
two-sided support orthogonality: $A^*B=AB^*=0$.  This rigidity, also
used in the recent treatment of Tingley's problem for Schatten
quasi-classes \cite{HuangZhu2026}, allows the max-phase data to detect
the mutual position of the left and right supports.  It provides a route
to rank-one geometry that does not rely on local convexity.

We investigate surjective mappings that preserve the max-phase kernel
$\Mp(A,B):=\max\{\norm{A+B}_p,\norm{A-B}_p\}$ on Schatten
quasi-classes.  Our main result gives a complete classification in finite
dimensions, without assuming continuity or linearity.  In the statement,
transpose and entrywise complex conjugation are taken relative to a fixed
orthonormal basis.

\begin{theorem}\label{thm:intro-main}
Let $0<p<1$ and $n\ge2$.  A map
$f:\mathcal S_p^n\to\mathcal S_p^n$ is surjective and satisfies
\begin{equation}\label{eq:maxidentity-intro}
 \Mp(f(A),f(B))=\Mp(A,B)\qquad(A,B\in\mathcal S_p^n)
\end{equation}
if and only if
\begin{equation}\label{eq:main-form-intro}
 f(A)=\eps(A)\Phi(A),
\end{equation}
where $\eps:\mathcal S_p^n\to\{-1,1\}$ is even and $\Phi$ is one of
$A\mapsto UAV$, $A\mapsto UA^{\mathsf t}V$,
$A\mapsto U\overline A V$, or $A\mapsto UA^*V$, with unitary matrices
$U,V$.  The value of $\eps(0)$ is immaterial.
\end{theorem}

The sign freedom in this theorem is unavoidable: changing the sign of
one argument leaves the max-phase kernel unchanged.  The theorem shows
that this pointwise ambiguity is the only obstruction to a standard
real-linear Schatten isometry.  In particular, no auxiliary regularity
condition is needed for the finite-dimensional classification.

The argument rests on three complementary ingredients.  First, equality
in the Schatten $p$-triangle inequality is used to recover two-sided
support orthogonality and identify rank-one operators intrinsically.
Second, a symmetrized spectral analysis of rank-one perturbations extracts
the transition quantity $|\Tr(e^*v)|$ from their max-phase profiles; this
reduces the rank-one action to a projective Wigner-type rigidity problem.
Third, we establish a max-profile uniqueness principle: an operator is
determined up to sign by the values of its max-phase kernel against all
real multiples of rank-one partial isometries.  Combining these steps
passes from rank-one rigidity to the classification of arbitrary matrices.
The profile argument is especially important below one, where a single
max-phase equality need not determine the direction of a vector.

We also consider arbitrary complex Hilbert spaces.  Without imposing
continuity, the same geometric analysis determines the action on the
finite-rank ideal and yields preservation of finite rank and two-sided
support orthogonality.  For continuous surjective max-phase isometries,
density of finite-rank operators then gives a full classification, with
a global rather than pointwise sign.  The remaining issue for general
discontinuous maps in infinite dimension is isolated in a conjecture,
whose resolution calls for a suitable extension of max-profile uniqueness
to infinite-rank operators.

The paper is organized as follows.  Section~\ref{sec:quasi-geometry}
develops the support geometry and rank-one phase-retrieval tools.
Section~\ref{sec:classification} proves the classification theorems in
finite dimensions and on the finite-rank ideal.  Section~\ref{sec:consequences}
discusses the geometric consequences, the continuous infinite-dimensional
case, and the outstanding general question.

\section{Support geometry and rank-one phase retrieval}
\label{sec:quasi-geometry}

Let $H$ be a nonzero complex Hilbert space.  For $0<p<1$, the Schatten
quasi-class is
\[
 \Sp(H)=\left\{A\in\mathcal K(H):
       \sum_{j\ge1}s_j(A)^p<\infty\right\},
 \qquad
 \norm A_p=\left(\sum_{j\ge1}s_j(A)^p\right)^{1/p}.
\]
We use the standard singular-value and trace-ideal framework of
\cite{McCarthy1967,Simon2005}; for the noncommutative $L_p$ setting and its
subspace geometry, see \cite{FackKosaki1986,RaynaudXu2003}.
It is complete for the translation-invariant metric
$d_p(A,B)=\norm{A-B}_p^p$.  The fundamental inequality is
\begin{equation}\label{eq:p-triangle}
 \norm{A+B}_p^p\le \norm A_p^p+\norm B_p^p.
\end{equation}
This is the Schatten form of the familiar subadditivity principle for
concave trace functions; compare \cite{Bourin2010}.

For $A\in\Sp(H)$, let $\suppL(A)$ and
$\suppR(A)$ denote the orthogonal projections
onto $\overline{\operatorname{ran}A}$ and
$\overline{\operatorname{ran}A^*}$, respectively.
Following the standard notion of disjoint
supports for operators \cite{BottazziConde2019},
we write
\begin{equation}\label{eq:support-orth}
	A\orth B
	\quad\Longleftrightarrow\quad
	A^*B=0\ \text{and}\ AB^*=0.
\end{equation}
Equivalently, the left supports of $A$ and $B$
are orthogonal, and so are their right supports.
For $0<p<1$, this relation admits a
characterization in terms of equality in the
$p$-triangle inequality, as established in
the following lemma.

The equality case of \eqref{eq:p-triangle} is the geometric replacement for
strict convexity.

The following characterization of support
orthogonality follows from
\cite[Lemma A.1]{HuangZhu2026}.
We record its Schatten-class formulation
for later use.

\begin{lemma}\label{lem:p-additivity}
	Let $0<p<1$ and $A,B\in\Sp(H)$. Then
	\begin{equation}\label{eq:equality-ptriangle}
		\norm{A+B}_p^p
		=\norm A_p^p+\norm B_p^p
		\quad\Longleftrightarrow\quad A\orth B.
	\end{equation}
	The same equivalence holds with $B$
	replaced by $-B$.
\end{lemma}

\begin{proof}
	The equivalence is a direct application of
	\cite[Lemma A.1]{HuangZhu2026} to the
	Schatten class $\Sp(H)$. The assertion
	for $-B$ follows by replacing $B$ with
	$-B$.
\end{proof}

\begin{definition}\label{def:maxphase}
A map $f:\Sp(H_1)\to\Sp(H_2)$ is a \emph{max-phase isometry} if
\[
 \Mp(f(A),f(B))=\Mp(A,B)\qquad(A,B\in\Sp(H_1)),
\]
where $\Mp$ is the max-phase kernel defined in the introduction.  No linearity,
continuity, or injectivity is included in the definition.
\end{definition}

\begin{proposition}\label{prop:norm-orth}
Let $0<p<1$ and let $f:\Sp(H_1)\to\Sp(H_2)$ be a surjective max-phase
isometry.  Then
\begin{enumerate}[label=\textup{(\roman*)}]
 \item $f(0)=0$ and $\norm{f(A)}_p=\norm A_p$ for every $A$;
 \item $A\orth B$ if and only if $f(A)\orth f(B)$;
 \item for $A\ne0$, the partially ordered family of orthogonality
 complements above $A^\perp$ is carried onto the corresponding family above
 $f(A)^\perp$.
\end{enumerate}
Here $A^\perp:=\{B\in\Sp(H_1):A\orth B\}$.
\end{proposition}

\begin{proof}
Putting $A=B=0$ in the max identity gives $f(0)=0$.  Putting $B=A$ gives
$2\norm{f(A)}_p=2\norm A_p$.

By \eqref{eq:p-triangle} and \Cref{lem:p-additivity},
\begin{equation}\label{eq:orth-max-character}
 A\orth B
 \quad\Longleftrightarrow\quad
 \Mp(A,B)^p=\norm A_p^p+\norm B_p^p.
\end{equation}
The max identity and norm preservation therefore prove (ii).

Finally, surjectivity and (ii) give
$f(A^\perp)=f(A)^\perp$.  If $A^\perp\subseteq B^\perp$ and
$W\in f(A)^\perp$, write $W=f(C)$.  Then $A\orth C$, hence $B\orth C$,
and therefore $f(B)\orth W$.  This proves
$f(A)^\perp\subseteq f(B)^\perp$; the converse follows in the same way.
\end{proof}

The order information in \Cref{prop:norm-orth} recognizes the atoms of the
operator geometry.

\begin{lemma}\label{lem:rankone-recognition}
Assume $\dim H\ge2$ and $0\ne A\in\Sp(H)$.  The following are equivalent.
\begin{enumerate}[label=\textup{(\roman*)}]
 \item $\rank A=1$.
 \item $A^\perp$ is maximal among the proper sets $B^\perp$ with $B\ne0$.
\end{enumerate}
Consequently, every surjective max-phase isometry preserves rank-one
operators in both directions.
\end{lemma}

\begin{proof}
Suppose first that $\rank A\ge2$.  Choose a rank-one partial isometry $e$
whose left and right supports are subprojections of $\suppL(A)$ and
$\suppR(A)$, respectively.  Every operator orthogonal to $A$ is orthogonal
to $e$, so $A^\perp\subseteq e^\perp$.  The inclusion is strict: choose a
second pair of initial and final unit vectors in the supports of $A$,
orthogonal to those of $e$, and form the corresponding rank-one operator.

Conversely, let $A$ have rank one and suppose
$A^\perp\subseteq B^\perp$ with $B\ne0$.  Testing the inclusion on
rank-one operators $\xi\otimes\eta$ with
$\xi\perp\suppL(A)H$ and $\eta\perp\suppR(A)H$ gives
\[
 \suppL(B)\le\suppL(A),\qquad \suppR(B)\le\suppR(A).
\]
Both right-hand projections have rank one, hence $B$ belongs to the same
one-dimensional complex corner as $A$, and $B^\perp=A^\perp$.  The final
assertion follows from \Cref{prop:norm-orth}(iii).
\end{proof}

The same order also detects every finite rank.  We shall use the following
form later.

\begin{lemma}\label{lem:finite-rank-recognition}
For $A\ne0$, $\rank A$ is the largest cardinality of a family
$e_1,\dots,e_m$ of mutually support-orthogonal rank-one partial isometries
satisfying $A^\perp\subseteq e_j^\perp$ for every $j$.  In particular, a
surjective max-phase isometry preserves every finite rank.
\end{lemma}

\begin{proof}
The inclusion $A^\perp\subseteq e^\perp$ is equivalent to
$\suppL(e)\le\suppL(A)$ and $\suppR(e)\le\suppR(A)$.  Thus an orthogonal
family of such $e$'s is exactly a pair of orthonormal families in the left
and right support spaces.  Its maximal cardinality is $\rank A$.  Apply
\Cref{prop:norm-orth,lem:rankone-recognition} to obtain invariance.
\end{proof}

We now develop the two technical ingredients needed to pass from support
geometry to general matrices.  Throughout the remainder of this section,
$0<p<1$.

Let $\Umin(H)$ denote the set of rank-one partial isometries; we use the
standard support-space terminology of \cite{HalmosMcLaughlin1963}.  If
$e,v\in\Umin(H)$, write
\[
 e=\xi\otimes\eta,
 \qquad v=\zeta\otimes\theta,
\]
with all four vectors of norm one, and set
\begin{equation}\label{eq:zD}
 z(e,v)=\Tr(e^*v),
 \qquad
 D(e,v)=\bigl(1-|\ip{\xi}{\zeta}|^2\bigr)
          \bigl(1-|\ip{\eta}{\theta}|^2\bigr).
\end{equation}
The number $D(e,v)$ is independent of the chosen vector representations.

For $0\le D\le1$ and $s\ge2\sqrt D$, define
\begin{equation}\label{eq:Phi}
 \Phi_{p,D}(s)=
 \left(\frac{s+\sqrt{s^2-4D}}2\right)^{p/2}
 +
 \left(\frac{s-\sqrt{s^2-4D}}2\right)^{p/2}.
\end{equation}

The following monotonicity properties
are contained in
\cite[Lemma 3.3 and the proof of
Lemma 3.4]{HuangZhu2026}.
We record them for later use.

\begin{lemma}\label{lem:Phi-monotone}
	For fixed $D$, the function
	$s\mapsto\Phi_{p,D}(s)$ is strictly
	increasing on $[2\sqrt D,\infty)$.
	For fixed $s=2$, the function
	$D\mapsto\Phi_{p,D}(2)$ is strictly
	monotone when $p\ne2$.
\end{lemma}

\begin{lemma}\label{lem:rankone-formula}
For $e,v\in\Umin(H)$ and $\lambda\in\T$,
\begin{equation}\label{eq:sym-rankone-formula}
 \Mp(\lambda e,v)^p
 =\Phi_{p,D(e,v)}
   \left(2+2\abs{\Re\bigl(\lambda\overline{z(e,v)}\bigr)}\right).
\end{equation}
Consequently, the function $\lambda\mapsto\Mp(\lambda e,v)$ determines
both $D(e,v)$ and $|\Tr(e^*v)|$.
\end{lemma}

\begin{proof}
Both $\lambda e-v$ and $\lambda e+v$ have rank at most two.  With respect
to orthonormal bases of $\spanop\{\xi,\zeta\}$ and
$\spanop\{\eta,\theta\}$, let $t_1,t_2$ be the eigenvalues of
$(\lambda e-v)^*(\lambda e-v)$.  A direct trace and determinant computation
(cf. \cite{HornJohnson2013}) gives
\begin{equation}\label{eq:trace-det}
 t_1+t_2=2-2\Re\bigl(\lambda\overline z\bigr),
 \qquad t_1t_2=D,
\end{equation}
where $z=z(e,v)$ and $D=D(e,v)$.  Thus
\[
 \norm{\lambda e-v}_p^p
 =\Phi_{p,D}\left(2-2\Re(\lambda\overline z)\right).
\]
Replacing $v$ by $-v$ changes the minus sign to a plus sign and leaves $D$
unchanged.  Since $\Phi_{p,D}$ is strictly increasing by
\Cref{lem:Phi-monotone}, taking the larger value gives
\eqref{eq:sym-rankone-formula}.

As $\lambda$ ranges over $\T$, the minimum and maximum of the right-hand
side are
\[
 \Phi_{p,D}(2),
 \qquad \Phi_{p,D}(2+2|z|),
\]
respectively.  The second part of \Cref{lem:Phi-monotone} recovers $D$
from the first number, and the first part then recovers $|z|$ from the
second.
\end{proof}

We next record the exact behavior on each complex rank-one line.

The following scalar case follows from
the max-phase classification for smooth
strictly convex real normed spaces
\cite{LiuLiuTan2025}.

\begin{lemma}\label{lem:scalar}
	Let $\varphi:\C\to\C$ be a surjective
	map satisfying
	\[
	\max\{|\varphi(z)+\varphi(w)|,
	|\varphi(z)-\varphi(w)|\}
	=
	\max\{|z+w|,|z-w|\}
	\]
	for all $z,w\in\C$. Then
	\[
	\varphi(z)=\delta(z)Rz,
	\]
	where $R$ is a real-linear orthogonal
	map of $\C$ and
	$\delta:\C\to\{-1,1\}$ is even.
	If $\varphi(1)=1$, then, after
	absorbing a global sign,
	$Rz=z$ or $Rz=\overline z$.
\end{lemma}

\begin{proof}
	Identify $\C$ with the real Euclidean
	plane. By \cite{LiuLiuTan2025},
	$\varphi(z)=\delta(z)Rz$ for a
	real-linear orthogonal map $R$
	and a sign function $\delta$.
	
	Surjectivity forces
	$\delta(-z)=\delta(z)$; otherwise,
	the two points $z$ and $-z$ would
	have the same image, leaving its
	antipode outside the range.
	
	If $\varphi(1)=1$, a global sign
	adjustment gives $R(1)=1$.
	The only real orthogonal maps of
	$\C$ fixing $1$ are the identity
	and complex conjugation.
\end{proof}

\begin{proposition}\label{prop:rankone-rigidity}
	Let $n\ge2$ and let
	$f:\mathcal S_p^n\to\mathcal S_p^n$
	be a surjective max-phase isometry.
	There exists a standard real-linear Schatten
	isometry $T$, either complex linear or
	conjugate linear, such that
	\begin{equation}\label{eq:rankone-sign}
		f(ze)\in\{T(ze),-T(ze)\}
		\qquad(z\in\C,\ e\in\Umin(\C^n)).
	\end{equation}
\end{proposition}

\begin{proof}
	Put $\mathcal U=\Umin(\C^n)$.
	By \Cref{lem:rankone-recognition,prop:norm-orth},
	$f$ maps rank-one operators onto rank-one
	operators and preserves their orthogonality
	complements. Since two nonzero rank-one
	operators have the same orthogonality
	complement precisely when they span the
	same complex line, $f$ maps each such
	line onto another. Thus \Cref{lem:scalar}
	gives
	\begin{equation}\label{eq:line-orientation}
		f(\lambda e)
		=\delta_e(\lambda)
		\lambda^{\sigma(e)}f(e),
		\qquad \lambda\in\T,
	\end{equation}
	where $\delta_e(\lambda)\in\{-1,1\}$
	and $\sigma(e)\in\{1,-1\}$.
	
	By \Cref{lem:rankone-formula} and
	surjectivity on each line, $f$ preserves
	$D(e,v)$ and
	\begin{equation}\label{eq:absolute-transition}
		|\Tr(f(e)^*f(v))|
		=|\Tr(e^*v)|,\qquad e,v\in\mathcal U.
	\end{equation}
	Writing $z=\Tr(e^*v)$ and
	$w=\Tr(f(e)^*f(v))$, the same formula
	yields
	\begin{equation}\label{eq:orientation-identity}
		|\Re(\overline\lambda\mu z)|
		=
		\left|\Re\left(
		\overline{\lambda^{\sigma(e)}}
		\mu^{\sigma(v)}w\right)\right|
	\end{equation}
	for all $\lambda,\mu\in\T$.
	
	If $z\ne0$ and $\sigma(e)\ne\sigma(v)$,
	setting $\lambda=\mu$ makes the left-hand
	side constant while the right-hand side
	varies with $\mu$, a contradiction.
	Since the graph joining pairs with
	nonzero trace is connected, $\sigma(e)$
	has a common value $\sigma$.
	Varying $\lambda$ then gives
	\[
	w\in
	\begin{cases}
		\{z,-z\},&\sigma=1,\\
		\{\overline z,-\overline z\},&\sigma=-1.
	\end{cases}
	\]
	
	It remains to remove these pair-dependent
	signs. Equip
	$Q=\mathcal U/\{\pm1\}$ with
	\[
	d_Q([e],[v])^2
	=2-2|\Re\Tr(e^*v)|.
	\]
	The preceding identity shows that
	$[e]\mapsto[f(e)]$ is a surjective
	isometry of $Q$, hence a homeomorphism.
	
	The manifold $\mathcal U$ is simply
	connected for $n\ge2$, as follows
	from its circle-bundle representation
	\[
	\mathcal U\cong
	(S^{2n-1}\times S^{2n-1})/\T.
	\]
	Consequently, the homeomorphism of
	$Q$ lifts through the double covering
	$\mathcal U\to Q$ to a homeomorphism
	$g:\mathcal U\to\mathcal U$ satisfying
	$g(e)=\pm f(e)$.
	
	On the set
	\[
	\Omega=\{(e,v)\in\mathcal U^2:
	\Tr(e^*v)\ne0\},
	\]
	the continuous function
	\[
	(e,v)\longmapsto
	\begin{cases}
		\displaystyle
		\frac{\Tr(g(e)^*g(v))}{\Tr(e^*v)},
		&\sigma=1,\\[6pt]
		\displaystyle
		\frac{\Tr(g(e)^*g(v))}
		{\overline{\Tr(e^*v)}},
		&\sigma=-1
	\end{cases}
	\]
	takes values in $\{-1,1\}$.
	The set $\Omega$ is path connected:
	for fixed $e$, the condition
	$\Tr(e^*v)\ne0$ defines a path-connected
	subset containing $e$, and the diagonal
	of $\mathcal U^2$ is connected.
	Since the displayed function equals
	$1$ on the diagonal, it is identically
	$1$. Together with
	\eqref{eq:absolute-transition}, this
	proves
	\[
	\Tr(g(e)^*g(v))
	=
	\begin{cases}
		\Tr(e^*v),&\sigma=1,\\
		\overline{\Tr(e^*v)},&\sigma=-1
	\end{cases}
	\]
	for all $e,v\in\mathcal U$.
	
	We may therefore apply
	\cite[Theorem 2]{Molnar2002},
	composing with entrywise conjugation
	when $\sigma=-1$.
	It follows that
	$g(e)=T(e)$ for a standard
	real-linear Schatten isometry $T$
	of the asserted orientation.
	Since $f(e)=\pm g(e)$, the scalar
	classification in \Cref{lem:scalar}
	extends this relation to every
	complex multiple of $e$, proving
	\eqref{eq:rankone-sign}.
\end{proof}

We next establish a max-profile uniqueness principle, which provides
the key passage from rank-one information to arbitrary matrices.  It is
precisely here that finite dimensionality removes the
singularity of the negative power $|A|^{p-1}$: after the supports are
identified, $A$ is invertible in its support corner.

\begin{lemma}
\label{lem:functional-phase-retrieval}
Let $E,F$ be complex Hilbert spaces of the same finite dimension, and let
$C,D:E\to F$ be invertible.  If
\begin{equation}\label{eq:functional-phase}
 |\Re\Tr(C^*e)|=|\Re\Tr(D^*e)|
\end{equation}
for every rank-one operator $e:E\to F$, then $D=C$ or $D=-C$.
\end{lemma}

\begin{proof}
Fix $\eta\in E$.  Operators of the form $\xi\otimes\eta$, with $\xi\in F$,
form a real linear space consisting of rank at most one.  On this space the
two sides of \eqref{eq:functional-phase} are the absolute values of two real
linear functionals.  Two real linear functionals with identical absolute
values are equal or opposite; hence
\[
 D\eta=\tau(\eta)C\eta,
 \qquad \tau(\eta)\in\{-1,1\}.
\]
If $\eta_1,\eta_2$ are linearly independent, invertibility of $C$ and the
identity for $\eta_1+\eta_2$ force
$\tau(\eta_1)=\tau(\eta_2)=\tau(\eta_1+\eta_2)$.  Collinear vectors are
handled directly.  Thus $\tau$ is constant and $D=\pm C$.
\end{proof}

\begin{theorem}\label{thm:max-profile}
Let $0<p<1$ and $A,B\in M_n(\C)$.  Suppose that
\begin{equation}\label{eq:profile-hyp}
 \Mp(A,te)=\Mp(B,te)
 \qquad(t\in\R,\ e\in\Umin(\C^n)).
\end{equation}
Then $B=A$ or $B=-A$.
\end{theorem}

\begin{proof}
Taking $t=0$ gives $\norm A_p=\norm B_p$.  If $A=0$, then $B=0$ and there
is nothing to prove; hence assume that both matrices are nonzero.  By
\eqref{eq:orth-max-character}, for a rank-one $e$,
\[
 A\orth e
 \quad\Longleftrightarrow\quad
 \Mp(A,e)^p=\norm A_p^p+1
 \quad\Longleftrightarrow\quad B\orth e.
\]
Thus $A$ and $B$ have the same rank-one orthogonality locus.  If their rank
is less than $n$, this locus is the Segre product
\[
 \mathbb P(\suppL(A)^\perp H)
 \times \mathbb P(\suppR(A)^\perp H),
\]
and equality of the two products gives
\begin{equation}\label{eq:same-supports}
 \suppL(A)=\suppL(B),\qquad \suppR(A)=\suppR(B).
\end{equation}
If the locus is empty, both matrices have full rank, and
\eqref{eq:same-supports} also holds.  Denote the common support spaces by
$F$ and $E$.  Restricted from $E$ to $F$, both $A$ and $B$ are invertible.

Write $A=u_A|A|$ and set
\[
 J_p(A)=u_A|A|^{p-1}:E\to F;
\]
define $J_p(B)$ similarly.  For a rank-one $e:E\to F$, finite-dimensional
trace differentiation gives
\begin{equation}\label{eq:directional-derivative}
 \left.\frac{d}{dt}\right|_{t=0}\norm{A+te}_p^p
 =p\Re\Tr(J_p(A)^*e).
\end{equation}
Indeed, this is the standard derivative of
$X\mapsto\Tr|X|^p$ at an invertible matrix; it follows either from singular
value perturbation or from the divided-difference formula for trace
functions.  Consequently,
\begin{equation}\label{eq:max-expansion}
 \Mp(A,te)^p
 =\norm A_p^p
  +p|t|\,|\Re\Tr(J_p(A)^*e)|+o(|t|).
\end{equation}
The identical expansion holds for $B$.  Comparing it with
\eqref{eq:profile-hyp} yields
\[
 |\Re\Tr(J_p(A)^*e)|
 =|\Re\Tr(J_p(B)^*e)|
\]
for every rank-one $e:E\to F$.  By
\Cref{lem:functional-phase-retrieval},
$J_p(B)=\pm J_p(A)$.

Finally, $J_p$ is injective up to the same real sign on the invertible
corner.  In fact, the polar part of $J_p(A)$ is $u_A$ and
$|J_p(A)|=|A|^{p-1}$, whence
\[
 A=u_A|J_p(A)|^{1/(p-1)}.
\]
It follows that $B=A$ in the plus case and $B=-A$ in the minus case.
\end{proof}

\begin{remark}\label{rem:why-profile}
One value of the max-kernel is not sufficient to recognize a ray below one.
For example, noncollinear unit vectors in $\ell_p^2$ can satisfy
$\max\{\norm{x+y}_p,\norm{x-y}_p\}=2$.  The strength of
\Cref{thm:max-profile} is the use of the entire one-parameter rank-one
profile, not a single equality case.
\end{remark}

\section{Classification theorems}
\label{sec:classification}

We can now prove the complete finite-dimensional result.

\begin{theorem}\label{thm:matrix-classification}
Let $0<p<1$ and $n\ge2$.  A map
$f:\mathcal S_p^n\to\mathcal S_p^n$ is a surjective max-phase isometry if
and only if there are a surjective real-linear Schatten isometry $T$ and an
even function $\eps:\mathcal S_p^n\to\{-1,1\}$ such that
\begin{equation}\label{eq:abstract-classification}
 f(A)=\eps(A)T(A)\qquad(A\in\mathcal S_p^n).
\end{equation}
\end{theorem}

\begin{proof}
Assume first that $f$ is a surjective max-phase isometry.  By
\Cref{prop:rankone-rigidity}, choose a real-linear Schatten isometry $T$ for
which $f(te)=\pm T(te)$ on every rank-one line.  Fix $A$ and put
$B=T^{-1}f(A)$.  For $e\in\Umin$ and $t\in\R$, the sign attached to $te$
only interchanges the two terms in a maximum.  Hence
\begin{align*}
 \Mp(B,te)
 &=\Mp(T(B),T(te))\\
 &=\Mp(f(A),f(te))
  =\Mp(A,te).
\end{align*}
By \Cref{thm:max-profile}, $B=\pm A$.  Thus
$f(A)=\eps(A)T(A)$ for some sign $\eps(A)$.

It remains to use surjectivity.  For $A\ne0$, the only possible preimages of
$T(A)$ under the displayed formula are $A$ and $-A$.  The two values
$f(A)$ and $f(-A)$ cover both $T(A)$ and $-T(A)$ exactly when
\[
 \eps(-A)=\eps(A).
\]
Since $f$ is onto, $\eps$ is even.

Conversely, let \eqref{eq:abstract-classification} hold with $\eps$ even.
For $A,B$ set $\alpha=\eps(A)\eps(B)\in\{-1,1\}$.  Then
\begin{align*}
 \Mp(f(A),f(B))
 &=\max\{\norm{T(A)+\alpha T(B)}_p,
          \norm{T(A)-\alpha T(B)}_p\}\\
 &=\max\{\norm{A+\alpha B}_p,\norm{A-\alpha B}_p\}
  =\Mp(A,B).
\end{align*}
To see surjectivity, let $C=T(A)$.  If $\eps(A)=1$, then $f(A)=C$; if
$\eps(A)=-1$, evenness gives $f(-A)=C$.
\end{proof}

The real-linear isometry in \Cref{thm:matrix-classification} has a concrete
operator form.  The classification of Schatten isometries goes back to
Arazy \cite{Arazy1975}; the quasi-Banach range is included in the formulation
of Erdos \cite{Erdos1994}.  It also follows from the recent unified solution
of Tingley's problem for $0<p\ne2<\infty$ \cite{HuangZhu2026}.

\begin{theorem}\label{thm:concrete}
Under the hypotheses of \Cref{thm:matrix-classification}, there exist
unitary matrices $U,V$ such that the isometry $T$ in
\eqref{eq:abstract-classification} is exactly one of
\begin{equation}\label{eq:fourforms}
 T(A)=UAV,\qquad UA^{\mathsf t}V,\qquad
 U\overline A V,\qquad UA^*V.
\end{equation}
Conversely, every map $f(A)=\eps(A)T(A)$ with $T$ as in
\eqref{eq:fourforms} and $\eps$ even is a surjective max-phase isometry.
\end{theorem}

\begin{proof}
A surjective real-linear isometry of $\mathcal S_p^n$, $p\ne2$, is either
complex linear or conjugate linear.  In the complex-linear case, the
Arazy--Erdos theorem gives $UAV$ or $UA^{\mathsf t}V$.  Composing a
conjugate-linear isometry with entrywise conjugation reduces it to the
complex-linear case and gives the other two forms.  The converse is contained
in \Cref{thm:matrix-classification} and is also immediate from unitary
invariance of the Schatten quasi-norm.
\end{proof}

The one-dimensional case is consistent with the theorem but is best stated
separately.

\begin{remark}\label{rem:scalar-case}
When $n=1$, the conclusion is \Cref{lem:scalar}: a surjective max-phase map
on $\C$, viewed as a real plane, is an even pointwise sign times a real
orthogonal map.  Rotations and reflections are precisely the scalar
instances of the complex-linear and conjugate-linear forms in
\eqref{eq:fourforms}.
\end{remark}

We next record what survives without finite dimensionality.  Let
$\Ffin(H)$ denote the finite-rank operators on $H$.

\begin{theorem}
\label{thm:finite-rank-core}
Let $H$ be a complex Hilbert space with $\dim H\ge2$, let $0<p<1$, and let
$f:\Sp(H)\to\Sp(H)$ be a surjective max-phase isometry.  Then there exists
a standard surjective real-linear Schatten isometry $T$ such that
\begin{equation}\label{eq:finite-rank-core}
 f(A)\in\{T(A),-T(A)\}
 \qquad(A\in\Ffin(H)).
\end{equation}
Moreover, $f$ preserves every finite rank, both singular support dimensions,
and two-sided support orthogonality on all of $\Sp(H)$.
\end{theorem}

\begin{proof}
The last assertion about orthogonality follows from
\Cref{prop:norm-orth}; finite-rank preservation follows from
\Cref{lem:finite-rank-recognition}.  The proof of
\Cref{prop:rankone-rigidity} uses only rank-two spectral computations, the
projective Wigner theorem, and connectedness of the rank-one partial
isometries, all valid on an arbitrary Hilbert space.  It therefore supplies
a standard $T$ for which \eqref{eq:rankone-sign} holds.

Fix $A\in\Ffin(H)$ and set $B=T^{-1}f(A)$.  Rank preservation makes $B$
finite rank.  The same calculation as in the proof of
\Cref{thm:matrix-classification} gives
$\Mp(A,te)=\Mp(B,te)$ for every rank-one $e$ and real $t$.  Equality of
rank-one orthogonality loci identifies the finite-dimensional left and right
support spaces of $A$ and $B$.  Restricting to their common finite support
corner, the proof of \Cref{thm:max-profile} applies verbatim and gives
$B=\pm A$.  This proves \eqref{eq:finite-rank-core}.
\end{proof}

\section{Consequences, regularity, and the infinite-dimensional boundary}
\label{sec:consequences}

The concrete classification immediately returns operator-theoretic
information which is invisible in the definition of the max-kernel.

\begin{corollary}\label{cor:singular-values}
Let $0<p<1$, $n\ge2$, and let $f$ be as in
\Cref{thm:matrix-classification}.  Then
\[
 s_j(f(A))=s_j(A)\qquad(1\le j\le n),
\]
and hence $\rank f(A)=\rank A$ for every $A$.
\end{corollary}

\begin{proof}
Each of the four maps in \eqref{eq:fourforms} preserves the full singular
value sequence, and multiplication by a real sign has no effect on singular
values.
\end{proof}

As established in \Cref{prop:norm-orth},
every surjective max-phase isometry
preserves two-sided support orthogonality.
In particular, if $A\orth B$, then
\begin{equation}\label{eq:orth-additivity}
	\norm{f(A)\pm f(B)}_p^p
	=\norm A_p^p+\norm B_p^p,
\end{equation}
by \Cref{lem:p-additivity} and norm
preservation.

In the finite-dimensional classification, pointwise signs are the only
source of discontinuity.  In arbitrary dimension, continuity eliminates
even this freedom.

\begin{theorem}\label{thm:continuous}
Let $H$ be a complex Hilbert space, $\dim H\ge2$, and $0<p<1$.  Suppose
$f:\Sp(H)\to\Sp(H)$ is a continuous surjective max-phase isometry.  Then
there are a standard surjective real-linear Schatten isometry $T$ and a
constant $\delta\in\{-1,1\}$ such that
\begin{equation}\label{eq:continuous-form}
 f(A)=\delta T(A)\qquad(A\in\Sp(H)).
\end{equation}
\end{theorem}

\begin{proof}
Let $T$ be supplied by \Cref{thm:finite-rank-core}.  On
$\Ffin(H)\setminus\{0\}$ write $f(A)=\eps(A)T(A)$.  Continuity implies that
$\eps$ is locally constant there: otherwise, a sequence $A_k\to A\ne0$
with $\eps(A_k)=-\eps(A)$ would make $f(A_k)$ converge simultaneously to
$T(A)$ and $-T(A)$ up to the fixed sign.  The punctured finite-rank ideal is
path connected.  Indeed, any two of its nonzero elements lie in a
finite-dimensional real subspace of dimension at least two, and a polygonal
path can be chosen to avoid the origin.  Hence $\eps$ is a single constant
$\delta$ on all nonzero finite-rank operators.

For arbitrary $A\in\Sp(H)$ choose finite-rank $A_k\to A$ in the Schatten
quasi-norm.  If $A\ne0$, continuity of $f$ and $T$ gives
\[
 f(A)=\lim_k f(A_k)=\delta\lim_k T(A_k)=\delta T(A).
\]
At zero the identity is automatic.
\end{proof}

\begin{remark}\label{rem:weaker-regularity}
The proof only uses continuity on the punctured class and sequential
continuity along finite-rank approximations.  Thus the conclusion remains
valid under that weaker pair of assumptions.  By contrast, measurability
alone cannot eliminate the phase: an even Borel sign function in
\eqref{eq:abstract-classification} still gives a Borel max-phase isometry.
\end{remark}

The next example makes clear why the smooth Banach-space argument cannot be
recycled below one.

\begin{example}\label{ex:line-test}
In $\ell_p^2$, put
\[
 x=2^{-1/p}(1,1),\qquad y=2^{-1/p}(1,-1).
\]
Then $\norm x_p=\norm y_p=1$, but
\[
 \max\{\norm{x+y}_p,\norm{x-y}_p\}=2^{1-1/p}<1.
\]
Thus the longer diagonal of two unit vectors may be shorter than either
vector.  In particular, equality in an ordinary triangle inequality cannot
be used to recover collinearity.  The proof of the main theorem succeeds
because it detects orthogonality through $p$-additivity and then uses the
entire rank-one perturbation profile.
\end{example}

We close by locating the exact remaining issue for arbitrary discontinuous
maps.  All constructions predicted by the finite-dimensional theorem remain
valid on $\Sp(H)$: if $T$ is a standard real-linear isometry and $\eps$ is
any even sign function, then $A\mapsto\eps(A)T(A)$ is a surjective max-phase
isometry.  \Cref{thm:finite-rank-core} proves that every surjective
max-phase isometry agrees with one of these constructions on a dense ideal.
For an infinite-rank compact operator, however, $|A|^{p-1}$ is unbounded;
the finite-corner differentiation in \Cref{thm:max-profile} therefore
requires a separate domain argument.  This is not a cosmetic point: small
singular values can turn first-order perturbations into $|t|^p$ terms.

The preceding results establish the classification
in finite dimensions and on the finite-rank ideal,
as well as on the full Schatten quasi-class under
the continuity assumption. The general
infinite-dimensional case without this assumption
remains open. A central difficulty is extending
the max-profile uniqueness principle in
\Cref{thm:max-profile} to arbitrary Schatten
operators. Finite-spectral approximation may
provide a useful approach to this problem.

\medskip
\noindent\textbf{Acknowledgments.}
The authors thank the anonymous referees for their valuable comments.
Thanks to all the members of the Functional Analysis Research team of the College of Mathematics and Statistics of Anqing Normal University for their discussion and correction of the difficulties and errors encountered in this paper.

\medskip
\noindent\textbf{Data availability.} No datasets were generated or analyzed in this work.\\
\textbf{Conflict of interest.} The author declares no conflict of interest.

\medskip
\noindent\textbf{Author contributions.}
All authors contributed equally to this work.

\end{document}